\documentclass[11pt]{article}

\usepackage{amsmath,amssymb,amsthm,mathtools}
\usepackage{booktabs}
\usepackage{longtable}
\usepackage{geometry}
\usepackage{graphicx}
\usepackage{microtype}
\usepackage{xcolor}
\usepackage{hyperref}

\hypersetup{
  colorlinks=true,
  linkcolor=blue!45!black,
  citecolor=blue!45!black,
  urlcolor=blue!45!black,
  pdftitle={A power-diagram upper bound for twenty-six variable-radius disks in a square},
  pdfauthor={Ayush Agarwal}
}

\newtheorem{theorem}{Theorem}[section]
\newtheorem{lemma}[theorem]{Lemma}
\newtheorem{proposition}[theorem]{Proposition}

\theoremstyle{definition}

\theoremstyle{remark}

\title{A Power-Diagram Upper Bound for Twenty-Six\\
Variable-Radius Disks in a Square}
\author{Ayush Agarwal\\
\small Independent Researcher, Bangalore, India\\
\small \href{mailto:ayushag1002@gmail.com}{\texttt{ayushag1002@gmail.com}}}
\date{24 August 2026}

\begin{document}
\maketitle

\begin{abstract}
Let $S_{26}$ denote the supremum of the sum of the radii of 26
pairwise interior-disjoint closed disks of positive radius contained in the
unit square.
We prove
\[
  2.635983084917607783\ldots
  \le S_{26}\le
  \sqrt{\frac{2577}{2500}+\frac{11}{\sqrt3}}
  =2.716919756099889566\ldots .
\]
The lower bound is accompanied by an exact rational feasibility
certificate. For the upper bound, we assign the disks their clipped power
cells, convert cell area into a reciprocal exterior-angle cost, and combine
this local inequality with Euler side counts. The extremal case for the
resulting relaxation has four boundary transitions and a four-cell boundary
cycle; its continuous angle allocation is reduced to a concave one-variable
problem. All finite numerical inequalities are rigorously enclosed by a
reproducible certificate. The argument does not determine $S_{26}$ exactly,
but gives a nontrivial global upper bound for a benchmark whose best-known
packing is asymmetric.
\end{abstract}

\section{Introduction}
\label{sec:introduction}

Let $Q=[0,1]^2$, and let $S_n$ denote the supremum of
$\sum_{i=1}^n r_i$ over all families of $n$ closed disks
$D(z_i,r_i)\subseteq Q$, with $r_i>0$, having pairwise disjoint interiors.
Unlike the
classical congruent-circle problem, both the centers and the radii are
variables.  This apparently modest change removes the reduction to a
single separation parameter and makes even numerical global
optimization substantially more difficult.  The sum-of-radii problem
has recently appeared as an explicit benchmark for mathematical
discovery systems~\cite{novikov2025alphaevolve,georgiev2025exploration},
and direct nonlinear formulations have been studied with modern global
optimization solvers~\cite{berthold2026revisited,berthold2026outofthebox}.

The case $n=26$ has received particular computational attention.  The
current Packomania table reports an asymmetric feasible packing with
sum of radii
\[
  2.635983084918,
\]
rounded there to the displayed 12 decimal places, and lists it as a
best-known construction rather than a proved optimum
~\cite{specht2026packomania}. In this paper we use an exact rational packing
that agrees with that entry to its reported precision; its sum is
\[
  L=2.63598308491760778318656948544348\ldots .
\]
All containment and non-overlap inequalities for this lower-bound
configuration are checked in exact rational arithmetic.

A general upper bound follows immediately from area and
Cauchy--Schwarz.  Disjointness gives
$\pi\sum_i r_i^2\leq 1$, and hence
\begin{equation}
  S_n\leq \sqrt{\frac{n}{\pi}}.
  \label{eq:intro-area-bound}
\end{equation}
This is also the basic dual bound used in recent global-optimization
work on the variable-radius model~\cite{berthold2026outofthebox}.  For
$n=26$, it gives
\[
  S_{26}\leq \sqrt{26/\pi}
  =2.876813695875796\ldots,
\]
leaving a gap of about $0.24083$ above the best-known construction.

Our main result replaces this area bound by
\begin{equation}
  \boxed{
  L\leq S_{26}\leq
  U:=\sqrt{\frac{2577}{2500}+\frac{11}{\sqrt{3}}}
  =2.716919756099889566\ldots .}
  \label{eq:intro-main-bracket}
\end{equation}
Thus the remaining certified gap is
$U-L=0.08093667118228\ldots$, or about $3.0705\%$ of the
lower bound.  Relative to~\eqref{eq:intro-area-bound}, the fraction of
the previous upper--lower gap removed is
\[
  \frac{\sqrt{26/\pi}-U}{\sqrt{26/\pi}-L}
  =0.663926\ldots,
\]
so the new estimate closes approximately $66.4\%$ of that gap. To our
knowledge as of 24 August 2026, the right endpoint in
~\eqref{eq:intro-main-bracket} is a new certified global upper bound
specifically for $n=26$
~\cite{specht2026packomania,berthold2026revisited,berthold2026outofthebox}.
It does \emph{not} prove that the current record construction is
optimal: determining the exact value of $S_{26}$ remains open.

The proof converts metric packing information into a planar
combinatorial budget.  We intersect the square with the power diagram
of the disks, a weighted Voronoi construction whose cells are convex
polygons~\cite{aurenhammer1987power,aurenhammer1991voronoi}.  Each power
cell contains its generating disk.  A local circumscribed-polygon
inequality, in the tradition of Fejes T\'oth's extremal polygon
methods~\cite{fejesToth1964regular,fejesToth2023lagerungen}, then turns
the exterior angles of a cell into a lower bound for its area.  Euler's
formula supplies a global side budget, while marked right-angle
contributions from the boundary of the square sharpen that budget.
The exceptional topology has exactly four open-side transitions; its
four corner cells form a boundary cycle, and retaining the transition
angles rather than merely counting them is what yields the stated
constant.  In this sense, broken boundary symmetry is not an incidental
feature of the construction but the configuration that the upper-bound
argument must treat most carefully.

Most of the argument is symbolic and geometric.  A short list of
one-variable trigonometric inequalities is certified by
outward-rounded interval arithmetic, while the lower construction is
verified over the rationals.  Validated numerics is a standard mechanism
for turning such finite computations into rigorous enclosures
\cite{moore2009interval}; interval-based computer-assisted proofs also
have substantial precedent in finite circle packing
\cite{markot2021interval}.  The accompanying verification code records
the intervals, endpoint conventions, and slack in every machine-checked
comparison.

The remainder of the paper develops the power-cell area inequality,
derives the planar side and boundary-angle budgets, treats the four-,
five-, and at-least-six-transition cases, and concludes with the exact
and interval certificates.  The construction and proof are logically
independent: the former establishes the left endpoint of
\eqref{eq:intro-main-bracket}, while the latter applies to every packing
of $26$ disks in the unit square.

\section{An exact feasible packing}\label{sec:construction}

The accompanying file \texttt{packing26\_exact.csv} contains 26 triples
$(x_i,y_i,r_i)$ with 50 digits after the decimal point. Every displayed
decimal is interpreted as an exact rational number, rather than as a
floating-point approximation.

\begin{proposition}[Exact lower certificate]\label{prop:lower-certificate}
The 26 disks specified in \textnormal{\texttt{packing26\_exact.csv}} have positive radii,
lie in $[0,1]^2$, have pairwise disjoint interiors, and satisfy
\[
  \sum_{i=1}^{26}r_i
  =2.63598308491760778318656948544348173039667679824847.
\]
Consequently, $S_{26}$ is at least the displayed value.
\end{proposition}

\begin{proof}
For each $i$, the four containment inequalities are
\[
 r_i\le x_i,\qquad r_i\le1-x_i,\qquad
 r_i\le y_i,\qquad r_i\le1-y_i.
\]
For each pair $i<j$, non-overlap is the polynomial inequality
\[
 (x_i-x_j)^2+(y_i-y_j)^2-(r_i+r_j)^2\ge0.
\]
The script \texttt{verify\_exact.py} parses the decimal strings as
fractions and verifies all $4\cdot26+\binom{26}{2}=429$ inequalities with
integer arithmetic. It also sums the 26 radii as an exact fraction. Thus
the calculation involves neither rounding nor a feasibility tolerance.
\end{proof}

The construction agrees, to the precision reported there, with the
current $N=26$ entry in the Packomania table~\cite{specht2026packomania}.
That table reports 78 contacts and no nontrivial symmetry. The construction
is used only as a lower bound here; no assertion of its global optimality is
made.

\begin{figure}[t]
  \centering
  \includegraphics[width=0.62\linewidth]{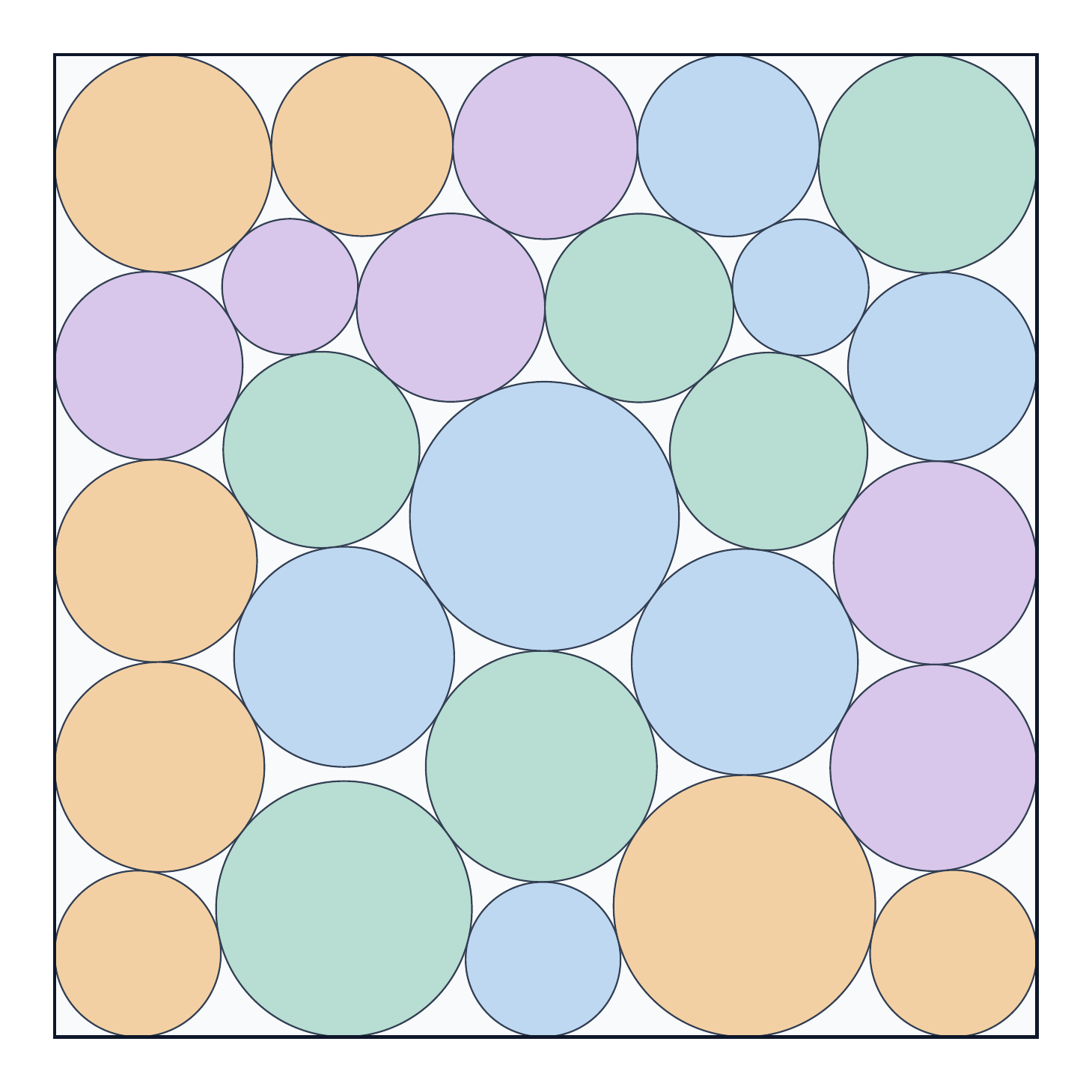}
  \caption{The exactly certified asymmetric 26-disk packing. The drawing is
  generated from the same decimal strings checked by
  \texttt{verify\_exact.py}.}
  \label{fig:packing}
\end{figure}

\section{The power-cell upper bound}\label{sec:upper-proof}

Throughout this section, \(Q=[0,1]^2\), and a packing means a family
\(\mathcal D=(D_i)_{i=1}^{26}\) of closed disks
\(D_i=B(z_i,r_i)\subset Q\), with \(r_i>0\), whose interiors are pairwise
disjoint. Put
\[
 S(\mathcal D)=\sum_{i=1}^{26}r_i,
 \qquad U^2=\frac{2577}{2500}+\frac{11}{\sqrt3}.
\]

\begin{theorem}\label{thm:global-upper-bound}
Every packing of \(26\) positive-radius disks in the unit square satisfies
\[
 S(\mathcal D)\le U=2.7169197560998895665\ldots .
\]
\end{theorem}

\subsection{Power cells and the local area inequality}

For \(x\in\mathbb R^2\), let
\[
 \pi_i(x)=|x-z_i|^2-r_i^2
\]
be the power of \(x\) with respect to \(D_i\), and define
\[
 P_i=\{x\in Q:\pi_i(x)\le \pi_j(x)\text{ for every }j\}.
\]
The difference \(\pi_i-\pi_j\) is affine, so \(P_i\) is the
intersection of \(Q\) with \(25\) affine half-planes. Moreover,
\begin{equation}\label{eq:disk-in-power-cell}
 D_i\subset P_i.
\end{equation}
Indeed, if \(x\in D_i\), then \(\pi_i(x)\le0\), whereas
\[
 |x-z_j|\ge |z_i-z_j|-|x-z_i|\ge r_i+r_j-r_i=r_j
\]
implies \(\pi_j(x)\ge0\). Thus every \(P_i\) is two-dimensional.
After redundant inequalities and consecutive collinear sides are
removed, its number \(m_i\) of sides satisfies
\begin{equation}\label{eq:m-range}
 3\le m_i\le25+4=29.
\end{equation}
The cells have disjoint interiors and cover \(Q\).

We use exterior angles: if a convex \(m\)-gon has interior angle
\(\pi-\beta_j\) at its \(j\)-th vertex, then
\(0<\beta_j<\pi\) and \(\sum_j\beta_j=2\pi\).

\begin{lemma}[Containing-polygon inequality]\label{lem:polygon}
Let \(P\) be a convex \(m\)-gon of area \(a\) containing a disk of
radius \(r>0\). If \(\beta_1,\ldots,\beta_m\) are its exterior angles,
then
\begin{equation}\label{eq:polygon-area}
 a\ge r^2\sum_{j=1}^m\tan\frac{\beta_j}{2}.
\end{equation}
\end{lemma}

\begin{proof}
Translate the disk centre to the origin. Write the supporting
half-planes of \(P\) as \(\langle n_j,x\rangle\le d_j\), with unit
outward normals \(n_j\). Containment gives \(d_j\ge r\). Replacing
every half-plane by \(\langle n_j,x\rangle\le r\) produces a polygon
\(P^{\rm tan}\subset P\). Since consecutive collinear sides were
removed, consecutive normal directions differ by an angle in
\((0,\pi)\); hence every new tangent line supports an actual side and
\(P^{\rm tan}\) has the same exterior angles.

At a vertex of exterior angle \(\beta_j\), each tangent segment to the
two adjacent contact points has length \(r\tan(\beta_j/2)\). Thus the
semiperimeter is \(r\sum_j\tan(\beta_j/2)\). A tangential polygon has
area equal to its inradius times its semiperimeter, proving
\eqref{eq:polygon-area}.
\end{proof}

Suppose that \(k\) specified exterior angles are at least \(\pi/2\).
Define
\begin{equation}\label{eq:ck-definition}
c_k(m)=
\begin{cases}
 m\tan(\pi/m),&m\le4,\\[1mm]
 k+(m-k)\tan\!\left(\dfrac{(4-k)\pi}{4(m-k)}\right),
 &m\ge5,\quad0\le k\le3,
\end{cases}
\qquad q_k(m)=\frac1{c_k(m)}.
\end{equation}
If \(k=4\), the angle sum forces the four marked exterior angles to be
exactly \(\pi/2\) and every other exterior angle to be zero. After
collinear sides are suppressed, \(m=4\); the polygon is a rectangle,
and we set \(c_4(4)=4\), \(q_4(4)=1/4\).

\begin{lemma}[Marked-angle cost]\label{lem:marked-cost}
For every feasible pair \((m,k)\),
\[
 \sum_{j=1}^m\tan\frac{\beta_j}{2}\ge c_k(m).
\]
\end{lemma}

\begin{proof}
Put \(y_j=\beta_j/2\), so \(\sum_jy_j=\pi\). For \(m\le4\), Jensen's
inequality gives \(\sum_j\tan y_j\ge m\tan(\pi/m)\).

Let \(m\ge5\) and \(1\le k\le3\); the case \(k=0\) is Jensen again.
Write the sum of the excesses of the marked variables above \(\pi/4\)
as \(E\ge0\). Jensen applied separately to the marked and unmarked
variables gives
\[
 \sum_j\tan y_j\ge
 k\tan\!\left(\frac\pi4+\frac Ek\right)
 +(m-k)\tan\!\left(
 \frac{(4-k)\pi}{4(m-k)}-\frac E{m-k}\right).
\]
The first mean is at least \(\pi/4\), while the second is at most
\((4-k)\pi/[4(m-k)]\le\pi/4\). The derivative in \(E\) is therefore
nonnegative, so the minimum occurs at \(E=0\). This gives
\eqref{eq:ck-definition}; the case \(k=4\) was handled above.
\end{proof}

Let \(a_i=|P_i|\), and let \(k_i\) be the number of marked angles of
\(P_i\). The preceding lemmas and Cauchy--Schwarz give
\begin{equation}\label{eq:master-cauchy}
 \left(\sum_i r_i\right)^2
 \le\left(\sum_i a_i\right)\left(\sum_iq_{k_i}(m_i)\right)
 =\sum_iq_{k_i}(m_i).
\end{equation}

\subsection{Generic topology and boundary runs}

We first prove the theorem for a generic clipped power diagram. Here
\emph{generic} means:
\begin{enumerate}
\item every interior power vertex is incident with exactly three cells;
\item every change of owner on the relative interior of a side of \(Q\)
is a simple two-cell transition: one power edge meets the side
transversely, and no third cell has equal power there;
\item no transition occurs at a corner of \(Q\), every corner has a
unique owner, and no boundary run has zero length;
\item redundant edges and degree-two collinear vertices are suppressed.
\end{enumerate}
Let \(T\) be the number of transitions on the open sides of \(Q\), and
let \(M=\sum_i m_i\).

\begin{lemma}[Euler side budget]\label{lem:euler-budget}
For a generic diagram of \(n\) nonempty cells,
\begin{equation}\label{eq:euler-budget}
 M=6n-2-T.
\end{equation}
Thus for \(n=26\), \(M=154-T\).
\end{lemma}

\begin{proof}
Let \(I\) be the number of interior vertices and \(E\) the number of
edges, including boundary edges. The connected planar graph has
\(I+T+4\) vertices and \(n+1\) faces, so Euler gives
\[
 E=I+T+n+3.
\]
Interior and open-side transition vertices have degree three, while
the four square corners have degree two. Hence
\[
 2E=3I+3T+8.
\]
Eliminating \(I\) gives \(E=3n+1\). There are \(T+4\) boundary edges.
An interior edge is counted twice in \(M\), and a boundary edge once;
therefore \(M=2E-(T+4)=6n-2-T\).
\end{proof}

\begin{lemma}[Boundary cells and runs]\label{lem:boundary-runs}
At least three distinct cells meet \(\partial Q\). If \(T>0\), the
transition points divide the boundary circle into exactly \(T\) open
runs, each with a constant owner. Consequently there are at most \(T\)
boundary cells, and the number of transition endpoints incident with
any fixed boundary cell is even.
\end{lemma}

\begin{proof}
Deleting \(T\) points from the circle \(\partial Q\) gives \(T\) runs.
Each run has two endpoints, so a cell is incident with twice its number
of runs.

One boundary cell would contain all four corners and hence their convex
hull \(Q\), leaving no room for another full-dimensional cell. Suppose
there were exactly two boundary cells \(A,B\). Their disjoint interiors
admit a weakly separating line \(L\), with \(A\subset H_A\) and
\(B\subset H_B\). Every point of \(\partial Q\) strictly in \(H_A\)
belongs to \(A\); by closure, so do the endpoints of \(L\cap Q\).
The polygon \(Q\cap H_A\) is the convex hull of these boundary points
and chord endpoints, hence lies in \(A\). Similarly
\(Q\cap H_B\subset B\). Thus \(A\cup B=Q\), again excluding the
remaining full-dimensional cells.
\end{proof}

At a simple boundary transition the two incident interior angles sum
to \(\pi\), so the two exterior angles also sum to \(\pi\). At least
one is therefore at least \(\pi/2\). A square corner contributes an
exterior angle exactly \(\pi/2\) to its unique owner.

\subsection{Concavity and the finite coefficient certificate}

Put
\begin{equation}\label{eq:gAB}
 g_m=\frac1{m\tan(\pi/m)},\qquad
 A=6g_5-5g_6,\qquad B=g_6-g_5,\qquad
 C=\frac{1/4-2g_5+g_6}{4}.
\end{equation}

\begin{lemma}[Concavity of \(g_m\)]\label{lem:g-concavity}
The function \(g(x)=1/[x\tan(\pi/x)]\) is strictly increasing and
strictly concave for \(x\ge3\). Consequently
\begin{equation}\label{eq:g-affine}
 g_m\le A+Bm\qquad(m\ge3).
\end{equation}
For a fixed number of integer side counts \(m_j\ge3\) and an upper
bound on their sum, \(\sum_jg_{m_j}\) is maximized by exhausting the
side budget and making the \(m_j\)'s differ by at most one.
\end{lemma}

\begin{proof}
Set \(u=\pi/x\). Then \(g(x)=u\cot u/\pi\), and \(g'(x)>0\) is
equivalent to \(u-\sin u\cos u>0\). Writing
\[
 h(u)=u^2\bigl(u\csc^2u-\cot u\bigr),
 \qquad g'(x)=\frac{h(u)}{\pi^2},
\]
direct differentiation gives
\[
 h'(u)=\frac{2u}{\sin^2u}
 \bigl(2u-u^2\cot u-\sin u\cos u\bigr)>0.
\]
Indeed, \(\tan u>u\) gives \(u^2\cot u<u\), and
\(\sin u\cos u<u\), for \(0<u\le\pi/3\). Since
\(du/dx=-u^2/\pi<0\), we have \(g''(x)<0\).

The line \(A+Bx\) passes through \((5,g_5)\) and \((6,g_6)\).
Decreasing first differences place every integer point \((m,g_m)\)
below its extension, proving \eqref{eq:g-affine}. If \(b\ge a+2\),
decreasing differences give
\(g_{a+1}+g_{b-1}\ge g_a+g_b\). Repeated smoothing, together with
monotonicity, proves the balancing assertion.
\end{proof}

The only computer-assisted ingredient is the following finite lemma.
It concerns all geometrically possible side counts by
\eqref{eq:m-range}; no configuration is sampled.

\begin{lemma}[Finite coefficient certificate]\label{lem:finite-certificate}
For
\[
 \mathcal F_{29}=
 \{(m,k):3\le m\le29,\ 0\le k\le\min(4,m),\
                 k=4\Rightarrow m=4\},
\]
one has
\begin{equation}\label{eq:old-plane}
 q_k(m)\le A+Bm+Ck.
\end{equation}
Equality occurs only at \((5,0),(6,0),(4,4)\); every other slack exceeds
\(2.8804896149\times10^{-5}\).

Define
\begin{equation}\label{eq:five-plane-constants}
 h_1=\frac1{1+4\tan(3\pi/16)},\qquad h_2=2-\sqrt3,
\end{equation}
\[
 C_5=h_2-h_1,\qquad
 B_5=g_6-2h_1+h_2,\qquad
 A_5=g_6-6B_5.
\]
Then \(B_5>0\), and for \(0\le k\le3\),
\begin{equation}\label{eq:five-plane}
 q_k(m)\le A_5+B_5m+C_5k\qquad(3\le m\le29).
\end{equation}
The only equalities are \((6,0),(5,1),(5,2)\). If \(k\ge1\) and
\((m,k)\notin\{(5,1),(5,2)\}\), the slack is at least
\begin{equation}\label{eq:delta-star}
 \delta_*:=A_5+4B_5+3C_5-\frac14
 >0.00155222863598066.
\end{equation}
Moreover, with
\begin{equation}\label{eq:delta-zero}
 \delta_0:=A_5+5B_5-g_5
 >0.00133076507398426,
\end{equation}
one has
\begin{equation}\label{eq:unmarked-saving}
 A_5+B_5m-g_m\ge\delta_0(6-m)\qquad(3\le m\le29),
\end{equation}
and \(2\delta_0>\delta_*\). Finally,
\begin{equation}\label{eq:slope-intervals}
\begin{aligned}
 0.00797077920&<g_7-g_6<0.00797077921,\\
 0.01339875050&<g_6-g_5<0.01339875051,
\end{aligned}
\end{equation}
and
\begin{equation}\label{eq:final-intervals}
\begin{aligned}
 26A_5+149B_5+9C_5-\delta_*&<7.380565,\\
 26A+148B+10C&<7.368670.
\end{aligned}
\end{equation}
In addition,
\begin{equation}\label{eq:global-comparison-intervals}
 2g_5+23g_6<7.190081,\qquad
 7.38165296<U^2<7.38165297.
\end{equation}
\end{lemma}

\begin{proof}
This is an exhaustive finite interval calculation. The dependency-free
companion certificate
\texttt{circle26/certify\_upper\_bound.py} uses exact rational
interval endpoints. It encloses \(\pi\) by Machin's formula and
alternating arctangent series, encloses sine and cosine by alternating
Taylor bounds, and encloses square roots by integer square-root
arithmetic. It loops over the full geometric range
\(3\le m\le29\) and every feasible \(k\). For each nonidentity in
\eqref{eq:old-plane} and \eqref{eq:five-plane}, it forms the interval
for the right-minus-left slack and requires its lower endpoint to be
positive. The six stated equalities follow algebraically.

For marked five-plane types, the same loop subtracts \(\delta_*\);
the minimum is attained at \((4,3)\). For unmarked types it checks
\eqref{eq:unmarked-saving} for \(m=3,4,5\); the statement is an
identity for \(m=5,6\), and for \(m>6\) follows from
\eqref{eq:five-plane} because its right side is negative. The certified
enclosures are
\[
\begin{array}{c|c}
\text{quantity}&\text{outward enclosure}\\ \hline
\delta_*&[0.0015522286359806639,0.0015522286359806640]\\
\delta_0&[0.0013307650739842652,0.0013307650739842654]\\
2\delta_0-\delta_*&[0.0011093015119878664,0.0011093015119878670]\\
B_5&[0.01206798542659390,0.01206798542659392].
\end{array}
\]
The same inclusion-isotone exact-rational evaluation gives
\eqref{eq:slope-intervals}--\eqref{eq:global-comparison-intervals}.
Thus the lemma is a finite list of strict interval inequalities, with
no floating-point decision or sampling. A successful run reports
\(618\) completed checks.
\end{proof}

\subsection{A cell owning two corners}

Assume one cell \(P_*\) owns at least two corners of \(Q\), and let
\(r=r_*\). Lemma~\ref{lem:boundary-runs} gives \(T\ge3\), while
\(m_*\ge3\). Lemma~\ref{lem:euler-budget} therefore gives
\[
 \sum_{i\ne *}m_i\le(154-3)-3=148.
\]
By \eqref{eq:g-affine},
\begin{equation}\label{eq:multicorner-qsum}
 \sum_{i\ne *}g_{m_i}
 \le25A+148B=2g_5+23g_6=:K.
\end{equation}

If \(P_*\) owns two adjacent corners, convexity makes it contain the
unit side between them. If the disk centre is at distance \(d\ge r\)
from that side, the cell also contains the disk point at distance
\(d+r\ge2r\); the resulting triangle has area at least \(r\). If it
owns only two opposite corners, use their diagonal, of length
\(\sqrt2\), as a base and the endpoints of a disk diameter perpendicular
to it. If the diameter straddles the diagonal, the two triangles have
total area \(\sqrt2r\); if it lies on one side, the triangle using its
farther endpoint has area at least \(\sqrt2r\). Thus
\begin{equation}\label{eq:multicorner-area}
 |P_*|\ge r.
\end{equation}

Apply \eqref{eq:master-cauchy} to the other \(25\) cells. Equations
\eqref{eq:multicorner-qsum} and \eqref{eq:multicorner-area} give
\[
 S(\mathcal D)\le r+\sqrt{K(1-r)}.
\]
Here \(K>23g_6=23/(2\sqrt3)>4\). Writing
\(t=\sqrt{1-r}\in[0,1]\), we have
\[
 \sqrt K-\bigl(r+\sqrt{K(1-r)}\bigr)
 =(1-t)(\sqrt K-1-t)\ge0.
\]
Consequently
\begin{equation}\label{eq:multicorner-bound}
 S(\mathcal D)\le\sqrt{2g_5+23g_6}
 =2.681432614083219\ldots<U,
\end{equation}
where the strict comparison follows from
\eqref{eq:global-comparison-intervals}.

It remains to consider the case in which the four corners of \(Q\) have
four distinct owners. The endpoints of each square side then have
different owners, so every side contains a transition and \(T\ge4\).

\subsection{Exactly four transitions}

Suppose \(T=4\). There is exactly one transition on each side of \(Q\),
and the boundary runs are precisely the four corner cells. Fix such a
cell with \(m\) sides. Let \(\alpha,\gamma\) be its exterior angles at
its two transitions, and put \(s=\alpha+\gamma\). Its square corner has
exterior angle \(\pi/2\), so \(0\le s\le3\pi/2\).

Put \(y_j=\beta_j/2\), \(t_j=\tan y_j\),
\(D=\sum_jt_j\), and \(\Phi=1/D\), on the simplex
\(\sum_jy_j=\pi\). If \(\sum_jv_j=0\), differentiation in direction
\(v\) gives
\[
 D'=\sum_jt_j^2v_j,\qquad
 D''=2\sum_jt_j(1+t_j^2)v_j^2.
\]
Cauchy--Schwarz gives
\[
 (D')^2\le
 \left(\sum_jt_j(1+t_j^2)v_j^2\right)
 \left(\sum_j\frac{t_j^3}{1+t_j^2}\right)
 \le\frac{DD''}{2}.
\]
Therefore
\begin{equation}\label{eq:Phi-concave}
 \Phi''=\frac{2(D')^2-DD''}{D^3}\le0.
\end{equation}
This extends to the boundary by continuity, with \(\Phi=0\) when some
\(y_j=\pi/2\).

Let \(\mathcal Y_m(s)\) be the compact convex set of vectors
\(0\le y_j\le\pi/2\) with sum \(\pi\), first two coordinates summing
to \(s/2\), and square-corner coordinate at least \(\pi/4\). The last
condition relaxes its actual equality and can only weaken the upper
bound. Define
\[
 q_m(s)=\max_{y\in\mathcal Y_m(s)}\Phi(y),\qquad
 I_m=\{s:\mathcal Y_m(s)\ne\varnothing\}.
\]
Partial maximization preserves concavity: convex combinations of
maximizers at \(s,s'\) are feasible at the corresponding convex
combination of \(s,s'\), and \eqref{eq:Phi-concave} applies.

Jensen first equalizes the transition coordinates and then the remaining
unconstrained coordinates. It gives
\begin{equation}\label{eq:Fms}
q_m(s)=\frac1{F_m(s)},\quad
F_m(s)=
\begin{cases}
2\tan(s/4)+(m-2)
 \tan\!\left(\dfrac{2\pi-s}{2(m-2)}\right),
 &s\le(6-m)\pi/2,\\[2mm]
1+2\tan(s/4)+(m-3)
 \tan\!\left(\dfrac{3\pi/2-s}{2(m-3)}\right),
 &s>(6-m)\pi/2.
\end{cases}
\end{equation}
The first case is the unconstrained minimizer; in the second the corner
constraint is active. The formulas agree at the switch. For \(m=3\),
\(I_3=[\pi,3\pi/2]\), so singular values of the first expression for
\(s<\pi\) are never used. For \(m\ge4\),
\(I_m\subset[0,3\pi/2]\).

Define
\begin{equation}\label{eq:Hm}
 H_m=\max_{s\in I_m}\left(q_m(s)+\frac{s-\pi}{50}\right).
\end{equation}

\begin{lemma}[Four-transition scalar bounds]\label{lem:H-bounds}
\[
 H_3<0.230,\quad H_4<0.2524,\quad H_5<0.263,\quad
 H_m<0.273\quad(m\ge6).
\]
Consequently
\begin{equation}\label{eq:H-envelope}
 H_m\le0.2100+0.0106\min(m,6)\qquad(m\ge3).
\end{equation}
\end{lemma}

\begin{proof}
Jensen gives \(q_3(s)\le g_3=1/(3\sqrt3)\), and \(s\le3\pi/2\).
Thus \(H_3\le g_3+\pi/100<0.230\), using
\(\sqrt3>17/10\) and \(\pi<22/7\).

For \(m=4\) and \(s\le\pi\), \eqref{eq:Fms} simplifies to
\(q_4(s)=\sin(s/2)/4\), so the expression in \eqref{eq:Hm} is at most
\(1/4\). For \(s=\pi+\delta\), put \(t=\tan(\delta/4)\). Then
\(0\le t\le\sqrt2-1<0.42\), and
\[
 q_4=\frac{(t-1)(t^2-2t-1)}{4(t+1)},\qquad
 \frac\delta{50}=\frac2{25}\arctan t\le\frac{2t}{25}.
\]
After clearing the positive denominator, it suffices that
\[
 P(t)=625t^3-1675t^2+194t-6<0.
\]
On \(0\le t\le0.1\),
\(P(t)\le-1612.5t^2+194t-6<0\), since the quadratic has negative
discriminant. On \(0.1\le t\le0.42\), \(P'(t)<0\) and
\(P(0.1)=-2.725\). Hence \(H_4<0.2524\).

For \(m=5\), put \(h_5(s)=q_5(s)+(s-\pi)/50\). This function is
concave. At the exact point \(s_5=(927296/10^6)\pi\), the
exact-rational interval certificate encloses the global tangent bound
\[
 h_5(s_5)+\frac{3\pi}{2}|h_5'(s_5)|
 <0.26214661<0.263.
\]
Every point of \(I_5\) is within \(3\pi/2\) of \(s_5\), and a tangent
line to a concave function is a global upper bound. Hence \(H_5<0.263\).
This calculation is included in
\texttt{circle26/certify\_upper\_bound.py}.

Let \(m\ge6\). Only the second branch of \eqref{eq:Fms} is needed.
The inequalities \(\tan x\ge x+x^3/3\) and \(\tan x\ge x\) give
\begin{equation}\label{eq:D-lower}
 F_m(s)\ge D(s):=1+\frac{3\pi}{4}+\frac{s^3}{96}.
\end{equation}
Put \(p(s)=D(s)-5s/4\). On \(0\le s\le3\pi/2\),
\[
 p'(s)=\frac{s^2}{32}-\frac54<0.
\]
Also \(p(0)>0\), while
\(p(3\pi/2)=1-9\pi/8+9\pi^3/256<0\), using
\(3<\pi<22/7\). Hence \(p\) has one zero \(s_*\). The derivative of
\[
 \psi(s)=\frac1{D(s)}+\frac{s-\pi}{50}
\]
has the sign of \(p(s)\), so \(\psi\) is maximized at \(s_*\). For
\(s_0=23/8\), \(\pi>333/106\) gives
\[
 p(s_0)=\frac{3\pi}{4}-\frac{115321}{49152}>0,
\]
so \(s_*>s_0\). At the critical point,
\[
 \psi(s_*)=\frac4{5s_*}+\frac{s_*-\pi}{50}.
\]
Its derivative is \(1/50-4/(5s^2)\), which is negative for
\(0<s\le3\pi/2\). Therefore
\[
 H_m\le\psi(s_*)<
 \frac4{5(23/8)}+\frac{23/8-\pi}{50}<\frac{273}{1000},
\]
again using \(\pi>333/106\). The four bounds imply
\eqref{eq:H-envelope}. At \(m=3,4,5\), its right-hand side is,
respectively,
\[
 0.2418,\qquad 0.2524,\qquad 0.2630.
\]
For every \(m\ge6\), its right-hand side is \(0.2736\).
\end{proof}

Let
\[
 R=\sum_{\text{four corner cells }i}\min(m_i,6),
 \qquad12\le R\le24.
\]
At each transition, adjacent exterior angles sum to \(\pi\); hence the
four values \(s_i\) sum to \(4\pi\). By \eqref{eq:Hm} and
\eqref{eq:H-envelope}, the corner contribution is at most
\begin{equation}\label{eq:corner-contribution}
 \sum_{\text{corner cells }i}q_{m_i}(s_i)
 \le0.84+0.0106R.
\end{equation}

The other \(22\) cells have at most \(150-R\) sides.
Lemma~\ref{lem:g-concavity} balances them and gives
\begin{equation}\label{eq:ordinary-balanced}
O(R)=
\begin{cases}
(R+4)g_6+(18-R)g_7,&R\le18,\\
(R-18)g_5+(40-R)g_6,&R\ge18.
\end{cases}
\end{equation}
For \(R\le18\), this is \(R+4\) hexagons and \(18-R\) heptagons; for
\(R\ge18\), it is \(R-18\) pentagons and \(40-R\) hexagons.
By \eqref{eq:slope-intervals}, the sum of
\eqref{eq:corner-contribution} and \eqref{eq:ordinary-balanced}
increases up to \(R=18\) and decreases afterwards. Thus
\begin{align}
 S(\mathcal D)^2
 &\le0.84+0.0106(18)+22g_6\notag\\
 &=1.0308+\frac{11}{\sqrt3}
 =\frac{2577}{2500}+\frac{11}{\sqrt3}=U^2.
\label{eq:T4-bound}
\end{align}

\subsection{Exactly five transitions}

Suppose \(T=5\). Mark the four square-corner angles and, at each
transition, choose one of the two incident exterior angles that is at
least \(\pi/2\). There are exactly nine marks. By
Lemma~\ref{lem:boundary-runs}, they are distributed among at most five
boundary cells.

\begin{lemma}\label{lem:T5-no-four-marks}
No cell has four marks when \(T=5\).
\end{lemma}

\begin{proof}
A cell with four marks has \(m=4\), by the angle sum following
\eqref{eq:ck-definition}. If it owns a square corner, then, because
the four corners have distinct owners, it owns exactly one corner and
its other three marks are transition vertices. Its total number of
incident transition endpoints is even by Lemma~\ref{lem:boundary-runs},
and hence is at least four. Together with the corner, these give at
least five distinct polygon vertices, contradicting \(m=4\).

If the cell owns no corner, its four marks are transition endpoints and
require at least two boundary runs. Each of the four distinct corner
owners requires a separate run, since a run cannot contain corners
with different owners. This requires at least six runs, whereas \(T=5\)
gives exactly five.
\end{proof}

Thus every cell has at most three marks, so
Lemma~\ref{lem:finite-certificate} applies. Since \(M=149\), \(B_5>0\),
and the total mark count is nine,
\begin{equation}\label{eq:V5}
 \sum_iq_{k_i}(m_i)
 \le V_5:=26A_5+149B_5+9C_5.
\end{equation}
If a marked cell is not one of the tight types \((5,1),(5,2)\),
\eqref{eq:delta-star} saves at least \(\delta_*\) from \(V_5\).

Suppose every marked cell is tight. At most five cells carry all nine
marks, and a tight marked cell carries at most two. Hence there are
exactly five marked pentagons: four carry two marks and one carries one.
They use \(25\) sides, leaving at most \(124\) sides for \(21\) unmarked
cells. Summing \eqref{eq:unmarked-saving} over those cells saves at least
\[
 \delta_0\left(6\cdot21-
       \sum_{\text{unmarked }i}m_i\right)
 \ge2\delta_0>\delta_*.
\]
In either case, \eqref{eq:master-cauchy}, \eqref{eq:V5}, and
\eqref{eq:final-intervals} give
\begin{equation}\label{eq:T5-bound}
 S(\mathcal D)^2\le V_5-\delta_*
 <7.380565<U^2.
\end{equation}

\subsection{Six or more transitions}

Suppose \(T\ge6\). Mark the four corner angles and one angle at every
transition. The number of marks is \(4+T\). No cell can have more than
four marked angles, and four marks force \((m,k)=(4,4)\). Hence
\eqref{eq:old-plane} applies cell by cell. Since \(M=154-T\),
\begin{equation}\label{eq:Tge6-affine}
 S(\mathcal D)^2\le26A+(154-T)B+(4+T)C.
\end{equation}
Here \(B>0\) by Lemma~\ref{lem:g-concavity}. Also \(C<0\): compare the
strict instance of \eqref{eq:old-plane} at \((4,0)\) with its equality
at \((4,4)\). Thus the right side decreases with \(T\). At \(T=6\),
\eqref{eq:final-intervals} gives
\begin{equation}\label{eq:Tge6-bound}
 S(\mathcal D)^2\le26A+148B+10C
 <7.368670<U^2.
\end{equation}
Equations \eqref{eq:multicorner-bound}, \eqref{eq:T4-bound},
\eqref{eq:T5-bound}, and \eqref{eq:Tge6-bound} prove the theorem for
generic diagrams.

\subsection{Removal of genericity}\label{subsec:remove-genericity}

\begin{lemma}[Density of generic feasible packings]\label{lem:generic-density}
Every packing of finitely many positive-radius disks is the limit of
strictly feasible packings whose clipped power diagrams satisfy the four
genericity conditions above.
\end{lemma}

\begin{proof}
Shrink every radius by a common factor \(1-\varepsilon\), keeping the
centres fixed. For every pair \(i\ne j\), the new pairwise slack is at
least \(\varepsilon(r_i+r_j)>0\), and the new wall slack is at least
\(\varepsilon r_i>0\). Strict feasibility is therefore open near the
shrunken parameter vector.

After removing the common term \(|x|^2\), the power diagram is the lower
envelope of
\[
 \ell_i(x)=-2\langle z_i,x\rangle+|z_i|^2-r_i^2.
\]
For \(r_i>0\), the map
\[
 (z_{i,1},z_{i,2},r_i)\longmapsto
 (-2z_{i,1},-2z_{i,2},|z_i|^2-r_i^2)
\]
has Jacobian determinant \(-8r_i\ne0\). It is a local diffeomorphism,
so arbitrary small perturbations of the affine coefficients arise from
arbitrary small perturbations of centres and positive radii.

For finitely many affine functions clipped by four fixed lines, a
nongeneric lower envelope satisfies one of finitely many determinant
equalities: four affine functions agree at one point; three agree at a
point of a prescribed wall; two agree at a square corner; a
pair-equality line coincides with a wall; or comparison lines coincide.
These conditions also contain every collapsed edge or boundary run:
collapsing an interior edge produces a four-function equality, while
collapsing a boundary edge produces a three-function wall equality or
a corner tie. Each determinant is a nonzero polynomial in the affine
coefficients (varying one intercept shows it is not identically zero).
A finite union of proper real algebraic zero sets has empty interior.
Thus generic coefficient vectors are dense. Pull one back through the
local diffeomorphism, choosing it within the open strict-feasibility
neighbourhood. Equation~\eqref{eq:disk-in-power-cell} ensures every
perturbed cell remains full-dimensional and nonempty.
\end{proof}

\begin{proof}[Proof of Theorem~\ref{thm:global-upper-bound}]
Let \(\mathcal D\) be arbitrary. Choose
\(\varepsilon_\nu\downarrow0\), shrink its radii by
\(1-\varepsilon_\nu\), and use Lemma~\ref{lem:generic-density} to make
a further generic perturbation small enough to preserve strict
feasibility and converge to \(\mathcal D\). The generic result gives
\(S(\mathcal D_\nu)\le U\) for every \(\nu\). Continuity of the sum of
radii yields \(S(\mathcal D)\le U\) in the limit.
\end{proof}

\section{The finite certificate}\label{sec:certificate}

The proof separates geometric statements from finite scalar inequalities.
The geometry and topology are proved in Section~\ref{sec:upper-proof}; the
accompanying program \texttt{certify\_upper\_bound.py} checks the following
finite claims:
\begin{enumerate}
  \item every affine reciprocal-angle inequality for all feasible pairs
  $(m,k)$, using the geometric bound $3\le m\le29$;
  \item monotonicity and discrete concavity of the finite sequence $g_m$
  on the required range, together with the two slope comparisons that select
  $R=18$ in the four-transition branch;
  \item the tangent enclosures for $H_4$ and $H_5$, and the rational
  checkpoints in the analytic $m\ge6$ estimate;
  \item the sign conditions, slack alternatives, and final totals in the
  multi-corner, five-transition, and at-least-six-transition branches.
\end{enumerate}

The exact lower certificate has no third-party dependencies. The upper
certificate records every outward enclosure used in the manuscript and exits
with a nonzero status if any comparison is undecided or false. From the
project directory, the complete audit is run by
\begin{verbatim}
python3 verify_exact.py
python3 certify_upper_bound.py
\end{verbatim}
The expected output and cryptographic hashes of the checked inputs are
recorded in Appendix~\ref{app:certificate-output}. The program supports the
proof but does not replace its combinatorial or geometric lemmas.

\section{Concluding remarks}\label{sec:conclusion}

The power-cell relaxation improves the elementary area bound substantially,
but it still leaves a gap of
\[
  2.716919756099889566\ldots
  -2.635983084917607783\ldots
  =0.080936671182281783\ldots .
\]
Closing this gap will require information not retained by the present
relaxation, such as compatibility between neighboring cell angles, metric
constraints on the power diagram, or a certified exclusion of alternative
contact topologies. The side-count and boundary-run arguments are not
specific to 26 disks; obtaining a useful uniform theorem for general $n$ is a
natural next question.

\appendix
\section{Certificate manifest}\label{app:certificate-output}

This appendix is generated after the final certificate run. It gives the
software version, SHA-256 hashes, and outward-rounded numerical intervals used
in the proof, so that the computational component can be reproduced without
relying on values copied from terminal output.

\begingroup
\small
\setlength{\tabcolsep}{5pt}
\renewcommand\UrlFont{\rmfamily\scriptsize}
\begin{longtable}{@{}p{0.30\textwidth}p{0.64\textwidth}@{}}
\toprule
Item & Certified value \\
\midrule
\endfirsthead
\toprule
Item & Certified value \\
\midrule
\endhead
Upper-certificate backend & \nolinkurl{exact fractions; Machin/arctan + Taylor trig + integer sqrt} \\
Completed upper checks & \nolinkurl{618} \\
Certified side domain & \nolinkurl{3 <= m <= 29} \\
Certificate SHA-256 & \nolinkurl{78a35735d5c4991e063ed8ae4c0b38bbd952fd56f9cf69d89474634de7ea9e98} \\
Coordinate SHA-256 & \nolinkurl{b966c6076f65598bdab5cf37680514329201b6a3e729d23b9cc0926a3c88fd44} \\
Exact verifier SHA-256 & \nolinkurl{e0ddc4f5b2b068f74825f86758a9778ff4880ceab1e9f88a945c02c1f755680c} \\
Exact sum of radii & {\raggedright\rmfamily\scriptsize 26359830849176077831865694854434817303966767\hspace{0pt}9824847/100000000000000000000000000000000000\hspace{0pt}000000000000000\par} \\
Minimum wall slack & \nolinkurl{1/1000000000000000000000000000000000000000000000000} \\
Minimum squared pair slack & {\raggedright\rmfamily\scriptsize 32294805931560370790816792913893725448064845\hspace{0pt}19247907/50000000000000000000000000000000000\hspace{0pt}00000000000000000000000000000000000000000000\hspace{0pt}000000000000000000000\par} \\
$\pi$ enclosure & \nolinkurl{[3.1415926535897932384626433832, 3.1415926535897932384626433833]} \\
$\sqrt3$ enclosure & \nolinkurl{[1.7320508075688772935274463415, 1.7320508075688772935274463416]} \\
$H_4$ tangent enclosure & \nolinkurl{[0.2522980788928591684028487608, 0.2522980788928591684028487609]} \\
$H_5$ tangent enclosure & \nolinkurl{[0.2621466099288132647684526720, 0.2621466099288132647684526721]} \\
$m\geq6$ analytic cap & \nolinkurl{[0.2729290164934215265350949584, 0.2729290164934215265350949585]} \\
Old-plane (5,1) slack & \nolinkurl{[0.0000288048961497346838133505, 0.0000288048961497346838133506]} \\
Marked gap delta-star & \nolinkurl{[0.0015522286359806639185304734, 0.0015522286359806639185304735]} \\
Unmarked gap delta-zero & \nolinkurl{[0.0013307650739842652871068354, 0.0013307650739842652871068355]} \\
Gap difference: twice delta-zero minus delta-star & \nolinkurl{[0.0011093015119878666556831975, 0.0011093015119878666556831976]} \\
Five-plane coefficient B5 & \nolinkurl{[0.0120679854265939093260256383, 0.0120679854265939093260256384]} \\
g7 - g6 enclosure & \nolinkurl{[0.0079707792012351987687883042, 0.0079707792012351987687883043]} \\
g6 - g5 enclosure & \nolinkurl{[0.0133987505005781746131324738, 0.0133987505005781746131324739]} \\
Multi-corner $S^2$ & \nolinkurl{[7.1900808638691657071380948085, 7.1900808638691657071380948086]} \\
Four-transition $S^2$ & \nolinkurl{[7.3816529610858834096006365855, 7.3816529610858834096006365856]} \\
Five-transition $S^2$ & \nolinkurl{[7.3805645675260637103662462842, 7.3805645675260637103662462843]} \\
Six-transition $S^2$ & \nolinkurl{[7.3686694114763682091431007492, 7.3686694114763682091431007493]} \\
Global $S$ upper enclosure & \nolinkurl{[2.7169197560998895665563886937, 2.7169197560998895665563886938]} \\
\bottomrule
\end{longtable}
\endgroup

\end{document}